\documentclass[10pt]{amsart}

\usepackage{amsmath, amsthm, amscd, amssymb, amsfonts, pstricks,tikz,bbm}
\usepackage[colorlinks,citecolor=red,pagebackref,hypertexnames=false]{hyperref}
\input amssym.def
\input amssym.tex

\newcommand{\Rb}{{\mathbb R}}

\newcommand{\Zz}{{\mathbb Z}}
\newcommand{\one}{{\mathbbm{1}}}

\newtheorem{thm}{Theorem}[section]
\newtheorem*{thm*}{Theorem}
\newtheorem{cor}[thm]{Corollary}
\newtheorem*{cor*}{Corollary}
\newtheorem{lem}[thm]{Lemma}

\newtheorem*{con*}{Conjecture}
\newtheorem*{prob*}{Problem}

\theoremstyle{definition}
\newtheorem{defn}[thm]{Definition}

\theoremstyle{remark}
\newtheorem{rem}[thm]{Remark}

\newtheorem{ex}[thm]{Example}

\begin{document}
\title{Metric trees of generalized roundness one}
\author{Elena Caffarelli}
\email{caffaree@math.rutgers.edu}
\address{Department of Mathematics, Rutgers University, Hill Center - Busch Campus, Piscataway, NJ 08854, USA}
\author[Ian Doust]{Ian Doust}
\email{i.doust@unsw.edu.au}
\address{School of Mathematics \& Statistics, University of New South Wales, Sydney, NSW 2052, Australia}
\author[Anthony Weston]{Anthony Weston}
\email{westona@canisius.edu}
\address{Department of Mathematics \& Statistics, Canisius College, Buffalo, NY 14208, USA}

\subjclass[2000]{46B20}

\keywords{generalized roundness, spherically symmetric trees}

\begin{abstract}
Every finite metric tree has generalized roundness strictly greater than one.
On the other hand, some countable metric trees have generalized roundness precisely
one. The purpose of this paper is to identify some large classes of countable
metric trees that have generalized roundness precisely one.

At the outset we consider spherically symmetric
trees endowed with the usual path metric (SSTs).
Using a simple geometric argument we show how to
determine reasonable upper bounds on the generalized roundness
of finite SSTs that depend only on the downward degree
sequence of the tree in question.
By considering limits,
it follows that if the downward degree sequence
$(d_{0}, d_{1}, d_{2}, \ldots)$ of a SST $(T,\rho)$ satisfies
$|\{ j \, | \, d_{j} > 1 \}|  =  \aleph_{0}$, then $(T,\rho)$
has generalized roundness one.
In particular, all
complete $n$-ary trees of depth $\infty$ ($n \geq 2$),
all $k$-regular trees ($k \geq 3$) and all inductive limits
of Cantor trees are seen to have generalized roundness one.

The remainder of the paper deals with two classes of countable metric
trees of generalized roundness one whose members are not, in general, spherically
symmetric. The first such class of trees are merely required to
spread out at a sufficient rate (with a restriction on the number of leaves)
and the second such class of trees resemble infinite combs. It remains an intriguing problem to completely classify
countable metric trees of generalized roundness one.
\end{abstract}
\maketitle

\tableofcontents

\section{Introduction and Basic Definitions}
Generalized roundness is a global geometric property of a metric space that may be related to isometric,
uniform and coarse classification schemes in certain settings. The notion was originally developed by
Enflo \cite{Enf} in order to study universal uniform embedding spaces. The ideas developed in \cite{Enf} were
later applied to the study of coarse embeddings into Hilbert space. This is the case in \cite{Dra} and \cite{Now}, for example.
In \cite{Len} it is shown that the generalized roundness and the supremal $p$-negative type of a metric space necessarily
coincide. In this way generalized roundness may be related to classical isometric embedding theory, and
\textit{vice-versa}. There is a long history here which is summarized in \cite{Pra}.
The relevant definition for our purposes is the following.

\begin{defn}\label{grdef}
The \textit{generalized roundness} $\gamma(X)$ of a metric space $(X,\delta)$ is the supremum of
the set of all $p \geq 0$ that satisfy the following property: For all integers $q \geq 2$ and all
choices of (not necessarily distinct) points $a_{1}, \ldots, a_{q}, b_{1}, \ldots, b_{q} \in X$,
\begin{eqnarray}\label{ONE}
\sum\limits_{1 \leq i < j \leq q} \bigl\{ \delta(a_{i},a_{j})^{p} + \delta(b_{i},b_{j})^{p} \bigl\}
& \leq & \sum\limits_{1 \leq i,j \leq q} \delta(a_{i},b_{j})^{p}.
\end{eqnarray}
The configuration of points $D_{q} = (a_{1}, \ldots, a_{q}; b_{1}, \ldots, b_{q}) \subseteq X$ underlying (\ref{ONE})
will be called a \textit{simplex} and the related function $\Bbbk_{D_{q}}(p)$, defined for all $p \geq 0$ by
\[
\Bbbk_{D_{q}}(p) = \sum\limits_{1 \leq i,j \leq q} \delta(a_{i},b_{j})^{p} -
\sum\limits_{1 \leq i < j \leq q} \bigl\{ \delta(a_{i},a_{j})^{p} + \delta(b_{i},b_{j})^{p} \bigl\},
\]
will be called the \textit{simplex gap function}.
\end{defn}

\noindent In relation to Definition \ref{grdef}, it is important to note that if $\gamma(X) < \infty$,
then the inequalities (\ref{ONE}) necessarily hold for exponent $p = \gamma(X)$. (See Remark 1.2 in \cite{Len}.)
We further note that if there is a simplex $D_{q} \subseteq X$ and a $\wp > 0$ such that $\Bbbk_{D_{q}}(\wp) < 0$,
then $\gamma(X) < \wp$. This fact follows from \cite[Theorem 2]{Sch} and \cite[Theorem 2.4]{Len}.

There has been recent interest in determining when the Cayley graphs
of finitely generated free groups have positive generalized roundness.
This stems from the connection between generalized roundness
and the coarse Baum-Connes conjecture given in \cite{Laf}. Arguing indirectly,
on the basis of relationships between kernels of negative type and equivariant Hilbert
space compression, Jaudon \cite{Jau} was able to prove that the generalized roundness of
the usual Cayley graph of a finitely generated free group (of rank $n \geq 2$) is one.
These graphs are $k$-regular trees (where $k = 2n$) and, as such, they provide specific
examples of countable spherically symmetric trees.

Recall that given a connected graph $G$ the usual path distance between two vertices $v$ and $w$ is the smallest
number of edges in any path joining $v$ to $w$. We shall call this the \textit{(unweighted) path metric} $\rho$
on $G$. Many planar graphs $(G, \rho)$, including all discrete trees, are known to have generalized roundness at least one.
In what follows, we shall denote the degree of a vertex $v$ in a graph $G$ by $\deg(v)$.

The connected graphs of interest in this paper
are trees with at most countably many vertices and where each vertex is of finite degree.
The condition on the degrees of the vertices is not a significant restriction as it follows immediately
from \cite[Theorem 5.6]{Dou} that if a tree $T$ has a vertex of infinite degree, then $\gamma(T, \rho) = 1$.
For brevity we shall use the term countable tree rather than countably infinite tree.

Many of our definitions will depend on the choice of a distinguished root vertex. We shall denote a tree
$T$ with metric $\delta$ and root $v_0$ by $(T,\delta,v_0)$. If $v_{0}$ or $\delta$ are clear in
a given setting, we may simply write $(T, \delta)$ or $T$.
For a rooted metric tree $(T, \delta, v_{0})$ we shall let $r(T) = r(T,v_0)$ denote the $v_{0}$-\textit{radius} of $T$.
In other words, $r(T, v_{0})$ is the maximal value of $\delta(v_0,v)$ if that quantity is finite and $\infty$ otherwise.

Our definition of spherically symmetric trees is based on the one in \cite{Lee}.

\begin{defn} Let $T$ be a tree endowed with the path metric $\rho$.
We shall say that $T$ is \textit{spherically symmetric} if
we can choose a root vertex $v_{0} \in T$ so that
\[ \rho(v_0,v_1) = \rho(v_0,v_2) \Rightarrow \deg(v_1) = \deg(v_2) \text{ for all } v_{1}, v_{2} \in T. \]
Such a triple $(T, \rho, v_{0})$ will be called a \textit{spherically symmetric tree} (SST).
\end{defn}

\noindent Some additional standard terminology and notation is helpful for describing SSTs.
Given two distinct vertices $v$ and $w$ in a tree $T$ with a designated root $v_{0}$,
we shall say that $w$ is a \textit{descendant} of $v$, and that $v$ is an \textit{ancestor} of $w$, if $v$ lies on the
geodesic joining $w$ and $v_0$. A descendant $w$ of $v$ for which $\rho(v,w) = 1$ is called a
\textit{child} of $v$. We shall denote by $d_\downarrow(v)$ the number of children of the vertex $v$.
A \textit{leaf} is a vertex $v$ of degree one.

Now suppose that $(T,\rho,v_0)$ is a SST.
Clearly, $d_\downarrow(v)$ depends only on $\rho(v_0,v)$. For $0 \le k < r(T)$,
let $d_k$ denote the integer such that $d_\downarrow(v) = d_k$ whenever $\rho(v_0,v) = k$.
The \textit{downward degree sequence} of $T$ is the finite or countable sequence
\[ d_{\downarrow}(T) = d_{\downarrow}(T,v_0) = (d_k)_{0 \le k < r(T)}. \]
Clearly one can reconstruct a SST directly from its downward degree sequence. It is worth noting however that a
countable tree can be a SST with respect to different root vertices. For example, the trees with downward
degree sequences $(2,2,1,2,1,2,1,2,\dots)$ and $(3,1,2,1,2,1,2,\dots)$ are graph isomorphic.

\begin{ex}
\end{ex}
\begin{center}
\begin{tikzpicture}[scale=0.7, dot/.style={draw,circle,inner sep=2.5pt}]
 \path node (v0) at (0,4) [dot] {}
        node (v11) at (-3,2) [dot] {}
        node (v12) at (0,2) [dot] {}
        node (v13) at (3,2) [dot] {}
        node (v21) at (-4,0) [dot] {}
        node (v22) at (-2.5,0) [dot] {}
        node (v23) at (-1,0) [dot] {}
        node (v24) at (1,0) [dot] {}
        node (v25) at (2.5,0) [dot] {}
        node (v26) at (4,0) [dot] {};
 \draw[line width = 1pt] (v0) to node {} (v11) to node {} (v21);
 \draw[line width = 1pt] (v11) to node {} (v22);
 \draw[line width = 1pt] (v0) to node {} (v12) to node {} (v23);
 \draw[line width = 1pt] (v12) to node {} (v24);
 \draw[line width = 1pt] (v0) to node {} (v13) to node {} (v25);
 \draw[line width = 1pt] (v13) to node {} (v26);
\end{tikzpicture}
\end{center}
\begin{center}
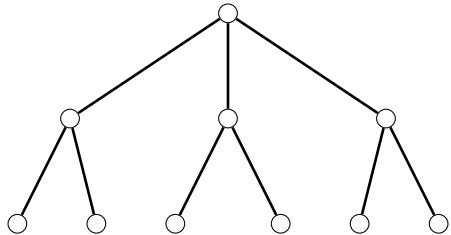
\textit{Figure 1.} The SST of radius 2 and downward degree sequence $(3,2)$.
\end{center}

\noindent The remainder of the paper is organized as follows.
In Section \ref{S2} we adopt a simple geometric approach to study the generalized roundness of SSTs.
The estimates for finite SSTs (Theorem \ref{SST1}) imply the existence of an uncountable class of countable SSTs,
all of whom have generalized roundness one. The sharp result in the case of countable SSTs (Theorem \ref{SST2})
is one of the main results of this paper. One advantage of our approach is that
it makes plain the geometry that underpins the Cayley graph result in \cite{Jau}. Section \ref{S3} generalizes
Theorem \ref{SST2} to include less regular countable metric trees. This is done by replacing spherical symmetry with
a more general spreading condition that we call infinite bifurcation. The main result in this direction is Theorem \ref{SST3}
which states that an infinitely bifurcating metric tree that has only finitely many leaves must have generalized roundness
one. In Section \ref{S4} we consider countable metric trees that resemble infinite combs. The situation becomes very
interesting when one starts knocking teeth out of the comb. The main result that we obtain in this direction is
Corollary \ref{cor:boundedgaps} which shows that if $C$ is an infinite comb graph whose distances between the teeth
is uniformly bounded by some constant $K > 0$, then $\gamma(C) = 1$. A number of the computations in this paper imply
explicit bounds on the generalized roundness of certain finite metric trees, including the complete binary tree $B_{m}$
of depth $m$ ($m \in \Zz^{+}$), and these are discussed in Section \ref{S5} together with some comparison bounds
that were obtained through numerical simulations.

\section{Spherically Symmetric Trees of Generalized Roundness One}\label{S2}

By \cite[Theorem 5.4]{Dou}, each finite metric tree has generalized roundness strictly greater than one.
Since Definition \ref{grdef} is predicated in terms of finite subsets
of the underlying metric space, it follows that each countable metric tree has generalized roundness at least one.
In fact, an explicit (non trivial) lower bound on the generalized roundness of a
finite metric tree can always be written down.
In the particular case of a finite SST $(T,\rho)$
it follows from \cite[Corollary 3.4]{Hli} that the generalized roundness $\gamma(T)$ of $(T,\rho)$ is at least
\begin{equation}\label{LowBnd}
1 + \Bigl\{ {\ln \Bigl( 1 + \frac{1}{{2 r(T)} \cdot (m-1) \cdot h(m)} \Bigl)}\Bigl/{\ln {2r(T)}} \Bigl\},
\end{equation}
where $m = |T|$ and $h(m) = 1 - \frac{1}{2} \cdot \bigl( {\lfloor \frac{m}{2} \rfloor}^{-1}
+ {\lceil \frac{m}{2} \rceil}^{-1} \bigl)$. Bounds such as this can be derived using nothing more
complicated than the method of Lagrange multipliers.

We now show how to obtain an upper bound on the generalized roundness of a finite SST $(T,\rho)$ that
behaves well if we let $m = |T| \rightarrow \infty$.

\begin{thm}\label{SST1}
Let $(T,\rho)$ be a finite SST with radius $n = r(T) \ge 3$ and with downward degree sequence $(d_{0}, d_{1}, \ldots, d_{n-1})$.
Then the generalized roundness $\gamma(T)$ of $(T,\rho)$ must satisfy
\begin{eqnarray}\label{UppBnd}
\gamma(T) & < & \min \frac{\ln\bigl(2 + \frac{2}{d_{0} \cdots d_{k} - 1}\bigl)}{\ln\bigl(2 - \frac{2k}{n}\bigl)},
\end{eqnarray}
where the minimum is taken over all integers $k$ such that $1 \leq k < n/2$ and $d_{0} \cdots d_{k} > 1$.
If no such integers $k$ exist, then $\gamma(T) \leq 2$ (with examples to show that equality is possible).
\end{thm}

\begin{proof}
Let $(T,\rho)$ be as in the statement of the theorem. Let $v_{0}$ denote the
root of $T$ and set $p = \gamma(T)$.

Consider an arbitrary integer $k$ such that $1 \leq k < n/2$ and $d_{0} \cdots d_{k} > 1$.
There are $d_{0}d_{1} \cdots d_{k-1}$ vertices at distance $k$ from $v_0$.
For each of the $d_{k}$ children of such a vertex, choose a leaf which is a descendant of that child.
This results in a total of $q = d_{0}d_{1} \cdots d_{k-1} d_{k}$ distinct leaves which we label $a_{1}, \ldots ,a_{q} $.
By the construction, the geodesic joining any two distinct leaves in this list must pass through at least one
of the vertices at distance $k$ from $v_0$. Consequently $\rho(a_{i},a_{j}) \geq 2(n-k)$ whenever $i \not= j$.
Now set $b_{1} = b_{2} = \cdots b_{q} = v_{0}$
and consider the corresponding generalized roundness inequality (\ref{ONE}). The left side of this
inequality is greater than $q(q-1)(2(n-k))^{p}/2$ and the right side is exactly $q^{2}n^{p}$, so we see that
$p$ must satisfy the weaker inequality $q(q-1)(2(n-k))^{p}/2 < q^{2}n^{p}$. On taking logarithms, elementary
rearrangement shows that
\begin{eqnarray}\label{TWO}
\gamma(T) & < & \frac{\ln\bigl(2 + \frac{2}{d_{0} \cdots d_{k} - 1}\bigl)}{\ln\bigl(2 - \frac{2k}{n}\bigl)}.
\end{eqnarray}
As (\ref{TWO}) holds for all $k$ such that $1 \leq k < n/2$ and $d_{0} \cdots d_{k} > 1$, the main assertion
of the theorem is evident. In the event that no such integers $k$ exist, the second assertion of the
theorem follows from the routine observation that the SST with degree sequence $(1,1)$ has generalized roundness two.
\end{proof}

\noindent

\noindent As we shall discuss in Section \ref{S5}, the upper bound given in Theorem \ref{SST1} is in no sense optimal.
It is however good enough to prove the following result.

\begin{thm}\label{SST2}
Let $(T,\rho)$ be a countable SST with degree sequence $(d_{0}, d_{1}, d_{2}, \ldots)$.
If $|\{ j \, | \, d_{j} > 1 \}|  =  \aleph_{0}$, then $(T,\rho)$ has generalized roundness one.
\end{thm}

\begin{proof}
Let $(T,\rho)$ be a given countable SST as in the statement of the theorem.
The condition on $(d_{0}, d_{1}, d_{2}, \ldots)$ ensures that $d_{1}d_{2} \cdots d_{k} \rightarrow \infty$
as $k \rightarrow \infty$. Moreover, by considering the truncations of $T$, it is clear that
$\gamma(T)$ satisfies (\ref{TWO}) for all $n \geq 3$ and all $k < n/2$. Now set $k = \lfloor \ln n \rfloor$
(for example) and let $n \rightarrow \infty$. We see immediately that $\gamma(T) \leq 1$. But, as we noted
at the outset of this section, $\gamma(T) \geq 1$ too. Thus, $\gamma(T) = 1$.
\end{proof}

\noindent One immediate consequence of Theorem \ref{SST2} is the following observation.

\begin{cor}
There exists an uncountable number of mutually non isomorphic countable SSTs
of generalized roundness one.
\end{cor}

\begin{proof}
Consider the set of degree sequences
\[ D = \{(1,d_1,d_2,d_3,\dots) \,: \, 1< d_1 < d_2 < d_3 < \dots \}.\]
This set is clearly uncountable, and by Theorem \ref{SST2} the SST corresponding to
each element of $D$ has generalized roundness one. Suppose that $(1,d_1,d_2,\dots)$
and $(1,d_1',d_2',\dots)$ are distinct elements of $D$, and let $T$ and $T'$ be the
corresponding SSTs. Let $k$ be the smallest integer such that $d_k \ne d_k'$ where,
without loss, we may assume that $d_k < d_k'$. Then $T$ has vertices with degree $d_k+1$,
but $T'$ does not, and so the two graphs are not isomorphic.
\end{proof}

\noindent Virtually all countable SSTs encountered in analysis, topology and discrete geometry
satisfy the condition placed on the downward degree sequence in the statement of Theorem \ref{SST2}. This is the
case for $k$-regular trees ($k \geq 3$), complete $n$-ary trees of depth $\infty$ ($n \geq 2$),
inductive limits of the Cantor trees described in \cite{Lee}, and so on.
In fact, only a countable number of countable SSTs fail the hypothesis of Theorem \ref{SST2}.
One example is the SST with downward degree sequence $(1,1,1,\ldots)$.
This countable SST clearly has generalized roundness two.
We do not know if the converse of Theorem \ref{SST2} is true.

\section{Infinitely Bifurcating Trees of Generalized Roundness One}\label{S3}
In this section we consider countable trees $S$, endowed with the usual path metric $\rho$, that
have at most finitely many leaves and develop the natural analogues of Theorems \ref{SST1} and \ref{SST2}.
This is done via the statement and proof of Theorem \ref{SST3}.
The main idea is that one can relax the hypotheses of the last section as long as the tree $S$
`spreads out' at a suitable rate. In particular, we will completely dispense with the requirement of
spherical symmetry.

Recall that any countable tree $(S,\rho)$ that has a vertex $v$ of degree $\aleph_{0}$ trivially has generalized
roundness one and so we will exclude such trees from the subsequent discussion.

\begin{defn}
Let $(S,\rho,v_0)$ be a countable rooted tree. A vertex $v \in S$ is said to be \textit{infinitely bifurcating}
if it has infinitely many descendants with vertex degree greater than or equal to $3$.
\end{defn}

\noindent Note that every ancestor of an infinitely bifurcating vertex is clearly also infinitely bifurcating,
and every infinitely bifurcating vertex has at least one child which is infinitely bifurcating.

\begin{lem}\label{biflem}
Let $(S,\rho,v_0)$ be a countable rooted tree. Then the following are equivalent.
\begin{enumerate}
\item $v_0$ is an infinitely bifurcating vertex.
\item There exists an infinitely bifurcating vertex $v \in S$.
\item There exist infinitely many bifurcating vertices in $S$.
\item There exist infinitely many vertices in $S$ with degree at least $3$.
\item There exists a radial geodesic path $(v_0,v_1,v_2,\dots)$ which contains
infinitely many vertices with degree at least $3$.
\end{enumerate}
\end{lem}

\begin{proof}
Most of the equivalences are trivial, or else easy consequences of the comments before the lemma.
To see that (1) implies (5), one may recursively construct a path by, for each $j \ge 0$, choosing $v_{j+1}$ to be
a child of $v_j$ which is also infinitely bifurcating. If this path contained only finitely many
vertices of degree at least $3$ there would be an integer $J$ so that $\deg(v_j) = 2$ for all $j \ge J$,
and this would contradict the fact that $v_J$ is infinitely bifurcating.
\end{proof}

\noindent A tree $(S,\rho,v_0)$ that satisfies any one (and hence all) of the conditions (1) through (5) of
Lemma \ref{biflem} will be said to be an \textit{infinitely bifurcating tree}.

\begin{thm}\label{SST3}
Let $(S,\rho,v_0)$ be a (countable) infinitely bifurcating tree with only finitely many leaves.
Then $(S,\rho,v_0)$ has generalized roundness one.
\end{thm}

\begin{proof}
Let $(S,\rho,v_0)$ be a countable tree that satisfies the hypotheses of the theorem.
Let $R = \max\{\rho(v_0,v) \,: \, \hbox{$v$ is a leaf}\}$. Since only finitely many vertices satisfy
$\rho(v_0,v) \le R$ we can choose a vertex $w$ with $\rho(v_0,w) = R+1$. The tree consisting of $w$
and all its descendants is then a subtree of $S$ which has no leaves and which still has an infinite number
of vertices with degree at least $3$.

It suffices to assume therefore that $(S,\rho,v_0)$ has no leaves.
The proof now proceeds by modifying the arguments used to establish Theorems
\ref{SST1} and \ref{SST2} together with the new structural ingredient of infinite bifurcation.

Since $S$ is infinitely bifurcating we can choose a radial geodesic path $P = (v_0,v_1,v_2,\dots)$ which
contains infinitely many vertices with degree at least $3$. Starting at $v_1$ (and preserving the original order),
denote the vertices in $P$ with degree at least $3$ by $c_1, c_2, c_3,\dots$.

Let $n\ge 2$ be a given integer and let $m = \rho(v_0,c_n)$. Clearly $m \ge n$. Suppose that $1 \le j \le n$.
Since $c_j$ has degree at least 3, it has a child which does not lie on $P$. Since $S$ has no leaves, we may choose a
descendant of this child, which we shall denote $a_j$ with $\rho(v_0,a_j) = m^2+m$. If we now fix $b_j = v_0$ for
$1 \le j \le n$, we see that $\rho(a_{i},b_{j}) = m + m^{2}$ for all $i$ and $j$, while $\rho(a_{i},a_{j}) \geq 2m^{2}$
whenever $i \not= j$ ($1 \leq i,j \leq n$).

Let $p = \gamma(S)$. It follows (by a similar argument to that
given in the proof of Theorem \ref{SST1}) that we must have
\[ n(n-1) (2m^{2})^{p} \leq 2n^{2} (m+m^{2})^{p}. \]
Thus
\begin{eqnarray}\label{THREE}
p & \leq & \frac{\ln \bigl(2 + \frac{2}{n-1}\bigl)}{\ln \bigl(2 - \frac{2}{m+1}\bigl)}.
\end{eqnarray}
If we now let $n$ (and therefore $m$) $\rightarrow \infty$ in (\ref{THREE}) we see that $p \leq 1$.
\end{proof}

\begin{rem}\label{SST4}
The hypothesis that an infinitely bifurcating tree $(S,\rho,v_0)$ have only finitely many leaves is clearly not
necessary in order that $\gamma(S) = 1$. We have already noted that if a countable tree has
vertex of degree $\aleph_{0}$, then its generalized roundness must be one. It is also clear from
the estimates in this paper that if a countable tree includes vertices of arbitrarily high degree, then
its generalized roundness must be one. It follows from either of these observations that a countable
metric tree with a countable number of leaves can have generalized roundness one. The
following example from \cite{Dou} exhibits this type of behavior nicely. Let $n \geq 2$ be an integer.
Let $Y_{n}$ denote the unique tree with $n+1$ vertices and $n$ leaves. In other words,
$Y_{n}$ consists of an internal vertex, which we will denote $r_{n}$, surrounded by $n$ leaves. We
endow $Y_{n}$ with the unweighted path metric $\rho$ as per usual. The generalized roundness of
$(Y_{n}, \rho)$ is computed explicitly in \cite[Theorem 5.6]{Dou}:
$$\gamma(Y_{n}) = 1 + \frac{\ln \bigl( 1+ \frac{1}{n-1}\bigl)}{\ln 2}.$$
We now form a countable tree $Y$ as follows: For each integer $n \geq 2$ connect $Y_{n}$
to $Y_{n+1}$ by introducing a new edge which connects the internal node $r_{n}$ of $Y_{n}$ to the
internal node $r_{n+1}$ of $Y_{n+1}$. We may further endow $Y$ with the unweighted path metric $\rho$.
The countable metric tree $(Y,\rho)$ has countably many leaves and it clearly has
generalized roundness one. However, the proof of Theorem \ref{SST3} does not apply to $(Y,\rho)$.
On the other hand it is easy to construct a countable tree with only finitely many leaves which has generalized
roundness $2$. For example, the positive integers with their usual metric will suffice.

These comments mean that, in spite of Theorem \ref{SST3}, the exact role that the
number of leaves plays in determining the generalized roundness of an infinitely bifurcating tree is still
not entirely clear. The interesting open question concerns the generalized roundness of infinitely bifurcating
trees (with infinitely many leaves) where the degrees of the vertices are bounded.
\end{rem}

\section{Comb Graphs Of Generalized Roundness One}\label{S4}

We now turn our attention to infinitely bifurcating trees with infinitely many leaves.
The simplest of these are graphs that resemble combs with infinitely many teeth. We
begin by considering the following example.

\begin{ex}\label{combs}
\end{ex}
\begin{center}
\begin{tikzpicture}[scale=0.7, dot/.style={draw,circle,inner sep=2.5pt}]
 \path node (v10) at (1,0) [dot] {}
        node (v11) at (1,1) [dot] {}
        node (v20) at (2,0) [dot] {}
        node (v21) at (2,1) [dot] {}
        node (v30) at (3,0) [dot] {}
        node (v31) at (3,1) [dot] {}
        node (v40) at (4,0) [dot] {}
        node (v41) at (4,1) [dot] {};
 \draw[line width = 1pt] (v10) to node {} (v20)
                              to node {} (v30) to node {} (v40)
                              to node {} (4.7,0);
 \draw[line width = 1pt] (v10) to node {} (v11);
 \draw[line width = 1pt] (v20) to node {} (v21);
 \draw[line width = 1pt] (v30) to node {} (v31);
 \draw[line width = 1pt] (v40) to node {} (v41);
 \draw (5,0) node[right] {$\dots$};
 \draw (1,-0.2) node[below] {$1$};
 \draw (2,-0.2) node[below] {$2$};
 \draw (3,-0.2) node[below] {$3$};
 \draw (4,-0.2) node[below] {$4$};
\end{tikzpicture}
\end{center}
\begin{center}
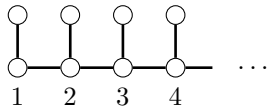
\textit{Figure 2.} The quintessential infinite comb tree $\mathcal{C}$.
\end{center}

\noindent We shall consider the vertices of $\mathcal{C}$ as points in the plane $\{(k,\ell) \,:\, k \in \Zz^+, \ell = 0,1\}$.
As per usual, all edges in this tree will be taken to be of length one.
For $n \ge 1$ we shall let $C_n$ denote the \textit{$n$-tooth comb} formed by taking the subtree of $\mathcal{C}$
containing the vertices $\{(k,\ell) \,:\, 1 \le k \le n, \ell = 0,1\}$. Given a subset $S \subseteq \Zz^+$, the tree $C_S$
is the subtree of $\mathcal{C}$ containing the vertices $\{(k,0) \,:\, k =1,2,3,\dots \} \cup \{(k,1) \,:\, k \in S\}$.
(Note that the term `comb graph' appears in the literature to describe slightly different types of graphs to
those considered here. Except when $S = \Zz^+$ the tree $C_S$ is not a comb graph in the sense of \cite{ABO}.)
Clearly in the trivial case $S = \emptyset$ we have $\gamma(C_\emptyset) = 2$. At the other extreme there is
enough structure in the quintessential comb $\mathcal{C}$ to show that $\gamma(\mathcal{C}) = 1$. In process of proving this result
(Corollary \ref{cor:quint}) we will motivate the more complicated arguments that close out this section.

\begin{thm}\label{C_n}
Suppose that $n \ge 2$. Then
\[ 1 + \left\{\ln\left(1+\frac{1}{2(n+2)(n-1)}\right)\Bigl/\ln(n+2) \right\}
\le \gamma(C_n) \le 1+ \frac{\ln 2}{\ln(n+1)- \ln 2}.\]
\end{thm}

\begin{proof}
The lower bound comes from \cite[Corollary 3.4]{Hli}. For the upper bound we construct a suitable simplex
$D_n = (a_1,\dots,a_n;b_1,\dots,b_n)$ in $C_n$ for which we can calculate the simplex gap function $\Bbbk_{D_{n}}(p)$.
By Definition \ref{grdef},
\[ \Bbbk_{D_n}(p) = \sum_{1 \le i,j \le n} \rho(a_i,b_j)^p
- \sum_{1 \le i < j \le n} \bigl\{ \rho(a_i,a_j)^p + \rho(b_i,b_j)^p \bigl\}.
\]
For $1 \le i \le n$ then, let $a_i = (i,0)$ and $b_i = (i,1)$ in $C_n$.

\medskip\noindent

\noindent \textbf{Claim.} $\Bbbk_{D_n}(p) = 1 + n 2^p - (n+1)^p$.

\medskip
We proceed by induction on $n$. One can check directly that $\Bbbk_{D_2}(p) = 1 + 2 \cdot 2^p - 3^p$ and so
the claim is true when $n=2$. Suppose then that for some $n > 2$ we have $\Bbbk_{D_{n-1}}(p) = 1 + (n-1) 2^p - n^p$.
By splitting off all the terms involving $a_n$ or $b_n$ we see that
\begin{multline*}
\Bbbk_{D_n}(p) = \Bbbk_{D_{n-1}}(p)
     + \left\{ \rho(a_n,b_n)^p + \sum_{i=1}^{n-1} \rho(a_i,b_n)^p + \sum_{i=1}^{n-1} \rho(a_n,b_i)^p \right. \\
       \left. - \sum_{i=1}^{n-1} \rho(a_i,a_n)^p - \sum_{i=1}^{n-1} \rho(b_i,b_n)^p \right\}.
\end{multline*}
Thus
\begin{align*}
\Bbbk_{D_n}(p) &= \Bbbk_{D_{n-1}}(p)
+ 1^p + (2^p + 3^p + \dots + n^p) + (2^p + 3^p + \dots + n^p) \\
& \qquad - (1^p + 2^p + \dots + (n-1)^p) - (3^p+4^p + \dots + (n+1)^p) \\
&= \Bbbk_{D_{n-1}}(p) + 2^p + n^p - (n+1)^p
\end{align*}
and hence, using the induction hypothesis,
\[
\Bbbk_{D_n}(p) = (1 + (n-1) 2^p - n^p) + (2^p + n^p - (n+1)^p) = 1 + n 2^p - (n+1)^p
\]
which proves the claim.

Note now that if
\[ p_n = 1+ \frac{\ln 2}{\ln(n+1)- \ln 2} = \frac{\ln(n+1)}{\ln(n+1)-\ln 2} \]
then a simple rearrangement shows that $(n+1)2^{p_n} = (n+1)^{p_n}$ and so
\[ \Bbbk_{D_n}(p_n) = 1 + n 2^{p_n} - (n+1)^{p_n}  < (n+1)2^{p_n} - (n+1)^{p_n} = 0 \]
and hence the generalized roundness of $C_n$ satisfies $\gamma(C_n) < p_n$ as required.
\end{proof}

\noindent It follows that any graph that contains copies of $C_n$ for arbitrarily large $n$ much have generalized roundness
equal to one. In particular, the quintessential comb $\mathcal{C}$ is seen to have generalized roundness one.

\begin{cor}\label{cor:quint}
$\gamma(\mathcal{C}) = 1$
\end{cor}

\noindent With considerably more effort one can apply similar techniques to other comb graphs which contain
sufficiently many teeth. We illustrate this phenomenon with the following theorem and corollary.

\begin{thm}\label{thm:ucomb}
Let $k \in \Zz^{+}$ be given and let $S_k = \{1, k+1, 2k+1, 3k+1, \ldots\}$. Then $\gamma(C_{S_k})=1$.
(In other words, comb graphs with uniform gaps of size $k$ between the teeth have generalized roundness one.)
\end{thm}

\begin{proof}
Let $k \in \Zz^{+}$ be given. Proceeding as in the proof of Theorem \ref{C_n}, given an integer $n \geq 2$,
we consider the simplex $D_n = (a_1, \ldots, a_n; b_1, \ldots, b_n) \subseteq C_{S_k}$ where,
for each $i$ ($1 \leq i \leq n$), $a_i = ((i-1)k+1,0)$ and $b_i = ((i-1)k+1,1)$.
  
Once again, using a simple induction argument, we can calculate the simplex gap function $\Bbbk_{D_{n}}(p)$.
By Definition \ref{grdef},
\begin{align}\label{eq:gap}
\Bbbk_{D_n}(p) = n - \sum_{i=1}^{n-1} (n-i)\{(ik)^p + (ik+2)^p - 2(ik+1)^p\}.
\end{align}
So if we set $\varphi_{p}(i) = (n-i)\{(ik)^p + (ik+2)^p - 2(ik+1)^p\}$, then
\[
\Bbbk_{D_n}(p) = n - \sum_{i=1}^{n-1}\varphi_{p}(i).
\]

The idea of the proof is to show that if $p \in (1,2]$ then there is always a large enough $n$
so that $$\sum_{i=1}^{n-1}\varphi_{p}(i) > n,$$ and hence $\Bbbk_{D_n}(p) < 0$, from
which it follows that $\gamma(C_{S_k}) = 1$. In fact, for $p = 2$, this occurs when
$n=3$. This is because $\Bbbk_{D_3}(2) = -1$ by (\ref{eq:gap}).
So we need only concentrate on the case when $p \in (1,2)$.

Assume, therefore, that $p \in (1,2)$. For fixed $i \in [1,n-1]$ and $(u,v) \in [0,1] \times [0,1]$,
we define $f(u) = (ik + u + 1)^p - (ik+u)^p$ and $g(u,v)=(ik+u+v)^{p-1}$. Then,
\begin{align*}
\varphi_{p}(i) &= (n-i)\{(ik)^p + (ik+2)^p - 2(ik+1)^p\} \\
               &= (n-i)\big\{ \{(ik+2)^p - (ik+1)^p\} - \{(ik+1)^p - (ik)^p\} \big\} \\
               &= (n-i)\{f(1)-f(0)\} \\
               &= (n-i)\int_{0}^{1} f'(u)du \\
               &= p(n-i)\int_{0}^{1}\{(ik + u + 1)^{p-1} - (ik+u)^{p-1}\} \, du \\
               &= p(n-i)\int_{0}^{1} \{g(u,1) - g(u,0)\} \, du \\
               &= p(n-i)\int_{0}^{1} \int_{0}^{1} g_{v}(u,v) \, dv \, du\\
               &= p(p-1)(n-i)\int_{0}^{1} \int_{0}^{1} (ik + u + v)^{p-2}\, dv \, du.
\end{align*}
In particular, elementary calculus shows that $(ik+u+v)^{p-2}$ achieves its minimum on the
compact set $[0,1] \times [0,1]$ at $(u,v) = (1,1)$, so in fact
\begin{align*}
\varphi_{p}(i) &\geq p(p-1)(n-i)\int_{0}^{1} \int_{0}^{1} (ik + 2)^{p-2}\, dv \, du \\
               &= p(p-1)(n-i)(ik + 2)^{p-2},
\end{align*}
and therefore
\begin{align}\label{eq:rsum}
\sum_{i=1}^{n-1} \varphi_{p}(i) \geq \, p(p-1)\sum_{i=1}^{n-1}(n-i)(ik+2)^{p-2}.
\end{align}
  
Now put $\psi_{p}(i) = (n-i)(ik+2)^{p-2}$. It is easy to see that
$\psi_{p}'(i) = k(p-2)(n-i)(ik+2)^{p-3} - (ik+2)^{p-2}$ and that, in particular,
$\psi_{p}'(i) < 0$ on $[1,n]$ because $p - 2 < 0$. Consequently, $\psi_{p}(i)$ is decreasing
on $[1,n]$, and so if we view the sum on the right hand side of \eqref{eq:rsum} as a Riemann
sum of rectangles with height $\psi_{p}(i)$ and width one, we see that
\begin{align*}
\sum_{i=1}^{n-1} \varphi_{p}(i) &\geq \, p(p-1)\sum_{i=1}^{n-1}(n-i)(ik+2)^{p-2} \\
                                &\geq \, p(p-1) \int_{1}^{n}(n-i)(ik+2)^{p-2} \, di \\
                                &= \, p(p-1) \frac{(ik+2)^{p-1}(k(np-pi+i)+2)}{k^2p(p-1)} \bigg|_{i=1}^{i=n} \\
                                &= \frac{1}{k^2} \big\{(nk+2)^{p-1}(nk+2) - (k+2)^{p-1}(k(np-p+1)+2) \big\} \\
                                &= \frac{1}{k^2} \big\{(nk+2)^{p} - (k+2)^{p-1}(k(np-p+1)+2) \big\} \\
                                &= \frac{n}{k^2} \left\{n^{p-1}\left(k+\frac{2}{n}\right)^{p}
                                   - (k+2)^{p-1}\left(k\left(p-\frac{p}{n} + \frac{1}{n}\right) + \frac{2}{n}\right) \right\}.
\end{align*}
(The cancellation of the $(p-1)$ term in the third line of the preceding computation is permitted, of course, because we are
assuming that $p - 1 > 0$.)
  
Now set
\begin{align*}
\digamma(n) = \left\{n^{p-1}\left(k+\frac{2}{n}\right)^{p}
- (k+2)^{p-1}\left(k\left(p-\frac{p}{n} + \frac{1}{n}\right) + \frac{2}{n}\right) \right\}.
\end{align*}
For $n$ large,  $\digamma(n) \approx n^{p-1}k^p - (k+2)^{p-1}kp$. Moreover,
$n^{p-1}k^p - (k+2)^{p-1}kp \rightarrow \infty$ as $n \rightarrow \infty$
because $p > 1$. Remembering that $k \in \Zz^{+}$ and $p \in (1,2)$ were fixed at the outset,
it follows that we may choose a large enough $n$ so that $\digamma(n) > k^2$. Then, we obtain
\begin{align*}
\sum_{i=1}^{n-1} \varphi_{p}(i) > \frac{n}{k^2} \cdot k^2 = n,
\end{align*}
and thus $\Bbbk_{D_n}(p) < 0$, as claimed at the outset.
Consequently, $\gamma(C_{S_k}) \notin (1,2)$.
Since we already noted that $\gamma(C_{S_k}) \neq 2$, it follows that $\gamma(C_{S_k}) = 1$, as desired.
\end{proof}

\noindent In fact, via the proof of Theorem \ref{thm:ucomb}, we can extend the class of comb graphs which have
generalized roundness one even further.

\begin{cor}\label{cor:boundedgaps}
Let $C_{S}$ be a comb graph whose distances between the teeth are uniformly bounded by some constant $K>0$.
Then $\gamma(C_{S})=1$.
\end{cor}

\begin{proof}
Let $C = C_S$ be a comb graph as per the hypotheses, and let $\{s_i\}_{i=1}^{\infty}$
be an enumeration of the elements in $S$ where the $s_i$ are inductively defined so that
$s_i < s_{i+1}$ for all $i \geq 1$. (Notice, then, that we have $s_{i+1} - s_{i} \leq K$
for all $i$ and hence, by induction, that $s_j - s_i \leq (j-i)K$ for all $i,j$ with $j > i$.)
  
Given an integer $n \geq 2$, consider the simplex $E_n = (a_1, \ldots, a_n; b_1, \ldots, b_n) \subseteq C$
where $a_i = (s_i,0)$ and $b_i = (s_i,1)$ for all $i$ ($1 \leq i \leq n$). Let $\Bbbk_{E_n}(p)$
denote the simplex gap of $E_n$. By Definition \ref{grdef},
\[
\Bbbk_{E_n}(p) = \sum_{i,j=1}^n \rho(a_i,b_j)^p - \sum_{1 \leq i < j \leq n} \{ \rho(a_i,a_j)^p + \rho(b_i,b_j)^p\}.
\]
  
In our particular case, it is easy to see that we have $\rho(a_i,b_j)^p = 1$ for all $i = j$
and that, for $i < j$, we have $\rho(a_i,b_j)^p = \rho(a_j,b_i)^p$. By an elementary
rearrangement of the terms, we further see that the expression for $\Bbbk_{E_n}(p)$ reduces to:
\begin{align*}
\Bbbk_{E_n}(p) &= n + \sum_{1 \leq i < j \leq n} \{2\rho(a_i,b_j)^p - \rho(a_i,a_j)^p + \rho(b_i,b_j)^p\} \\
               &= n - \sum_{1 \leq i < j \leq n} \{\rho(a_i,a_j)^p + \rho(b_i,b_j)^p - 2\rho(a_i,b_j)^p\} \\
               &= n - \sum_{i=1}^{n-1} \sum_{j=i+1}^{n} \{\rho(a_i,a_j)^p + \rho(b_i,b_j)^p - 2\rho(a_i,b_j)^p\} \\
               &= n - \sum_{i=1}^{n-1} \sum_{j=i+1}^{n} \{\rho(a_i,a_j)^p + (\rho(a_i,a_j)+2)^p - 2(\rho(a_i,a_j)+1)^p\} \\
               &= n - \sum_{i=1}^{n-1} \sum_{j=i+1}^{n} \{\varkappa^p + (\varkappa+2)^p - 2(\varkappa+1)^p\},
\end{align*}
where $\varkappa = \rho(a_{i},a_{j}) = s_j - s_i \leq (j-i)K$ for all $j > i$. Moreover, it is easy to verify
that the function $f(x) = x^p + (x+2)^p - 2(x+1)^p$ is non increasing on the interval $[0,\infty)$ provided $1 \leq p \leq 2$.
(For example, just note that the function $y = x^{q}$ is concave down on the interval $[0,\infty)$ provided $0 < q < 1$.)
Bearing these considerations in mind, it therefore follows that
\begin{align}
\Bbbk_{E_n}(p) &\leq n - \sum_{i=1}^{n-1} \sum_{j=i+1}^{n} \{((j-i)K)^p + ((j-i)K+2)^p - 2((j-i)K+1)^p\} \nonumber \\
                &= n - \sum_{i=1}^{n-1} (n-i) \{(iK)^p + (iK+2)^p - 2(iK+1)^p\}. \label{eq:iK}
\end{align}
But the upper bound (\ref{eq:iK}) on $\Bbbk_{E_{n}}(p)$ is exactly the quantity $\Bbbk_{D_n}(p)$ we
calculated in \eqref{eq:gap} when we considered a comb graph with uniform gaps of size (little) $K$.
So by repeating the proof of Theorem \ref{thm:ucomb} verbatim, it follows that if $p$ is any number
in the interval $(1,2]$, then there is $n$ large enough so that $\Bbbk_{E_n}(p) \leq \Bbbk_{D_n}(p) < 0$,
in which case $C$ does not have generalized roundness $p$. Therefore, it must be
the case that $\gamma(C)=1$.
\end{proof}

\begin{rem}
One may clearly develop more general versions of Theorems \ref{SST1}, \ref{SST2}, \ref{SST3}, \ref{thm:ucomb}
and Corollary \ref{cor:boundedgaps}. For example, there is actually no need to restrict the size of the gap $k$
to the positive integers in our statement of Theorem \ref{thm:ucomb}. Indeed, upon careful examination of the
proof, it is clear that any real number $k$ will work, so long as it is positive. Other variants of the theorems
in this paper may be formulated and proven for path weighted trees and certain $\mathbb{R}$-trees. In many
such settings the calculations become much more cluttered but not necessarily more complicated. The reader may care
to formulate their own theorems along these lines.
\end{rem}

\section{Some Comments on the Bounds}\label{S5}
Determining the precise value of the generalized roundness of even relatively simple finite metric trees
has proven to be a difficult non linear problem. Recent work of S{\'a}nchez \cite{San} (following Wolf \cite{Wol})
provides an expression for the generalized roundedness of a finite metric tree $(X = \{x_i\}_{i=1}^n,\rho)$ in terms
of the matrices $A_p$ where $A_p[i,j] = \rho(x_i,x_j)^p$. S{\'a}nchez shows that 
\begin{equation}\label{sanchez}
\gamma(X) = \inf \{p \ge 0 \,:\, \hbox{$\mathop{\mathrm{det}}(A_p) = 0$ or $\langle A_p^{-1}\one,\one \rangle = 0$}\} 
\end{equation}
where $\one = (1,1,\dots,1) \in \Rb^n$. For relatively small trees, this expression can be calculated, at least numerically.

The results of this paper allow us to give upper and lower bounds for standard families
of finite metric trees and it is interesting to compare these bounds with those given by the above formula. 
Recall that $C_{n}$ denotes the $n$-tooth comb graph that was introduced in Example \ref{combs}.
The following table compares the value of $\gamma(C_n)$, found numerically using (\ref{sanchez}),
with the upper and lower bounds from Theorem \ref{C_n}.
\smallskip
\begin{center}
\begin{tabular}{c|c|c|c}
        &  Lower bound      & $\gamma(C_n)$           & Upper bound \\
   $n$  &  -- Theorem \ref{C_n} &  -- Using (\ref{sanchez})  & -- Theorem \ref{C_n} \\[0.3ex]
   \hline\\[-2ex]
            2& 1.084962501& 2.000000000& 2.709511290\\
            3& 1.030315033& 1.775743466& 2.000000000\\
            4& 1.015291659& 1.494625215& 1.756470798\\
            5& 1.009095783& 1.445567766& 1.630929754\\
            6& 1.005973969& 1.410423534& 1.553294756\\
            7& 1.004194680& 1.383890448& 1.500000000\\
            8& 1.003091077& 1.363024724& 1.460845421\\
            9& 1.002362796& 1.346093176& 1.430676558\\
           10& 1.001858801& 1.332013004& 1.406598009\\
           15& 1.000740727& 1.285796898& 1.333333333\\
           20& 1.000386749& 1.259241515& 1.294783735\\
           25& 1.000234025& 1.241453867& 1.270238154\\
           30& 1.000155421& 1.228453930& 1.252895891\\
           35& 1.000110049& 1.218402824& 1.239812467\\
           40& 1.000081656& 1.210319687& 1.229486647\\
           45& 1.000062790& 1.203627556& 1.221064730\\
           50& 1.000049658& 1.197962011& 1.214021611\\

  \hline
\end{tabular}
\end{center}
\smallskip

\noindent It is also interesting to consider the case of binary trees.
For $m \ge 2$, let $(B_{m},\rho)$ denote the complete binary tree of depth $m$.
Using (\ref{sanchez}) in this setting quickly becomes more difficult because the size of
the matrices involved grow exponentially fast.
Since $B_m$ contains a copy of $C_m$, we can obtain upper bounds for $\gamma(B_m)$ using both
Theorem \ref{SST1} and Theorem \ref{C_n}. It should be noted that even in this relatively simple case,
it is not easy to precisely identify the value of $k$ at which the minimum in
(\ref{UppBnd}) occurs. However, elementary estimation
(confirmed by numerical calculations) shows that the minimum occurs when $k$ is approximately
\[ \frac{\ln m + \ln \ln 2}{\ln 2} - 1. \]

Numerical calculation of the upper bound given by Theorem \ref{SST1} shows that,
unlike the case for the comb graphs given above, the upper bound appears to rather far from being sharp. Indeed
for small values of $m$, the upper bound given by (\ref{UppBnd}) can be greater than the trivial upper bound $2$.
The following table compares $\gamma(B_m)$ (again found numerically using (\ref{sanchez})) with the upper and lower
bounds from (\ref{LowBnd}), (\ref{UppBnd}) and Theorem \ref{C_n}.

\begin{center}
\begin{tabular}{c|c|c|c|c}
        &  Lower bound      & $\gamma(B_m)$     & Upper bound    & Upper bound \\
   $m$  &  -- Using (\ref{LowBnd}) &  -- Using (\ref{sanchez}) & -- Theorem \ref{C_n} & -- Using (\ref{UppBnd}) \\[0.3ex]
   \hline\\[-2ex]
   2  &  1.0412    &  1.5272    &  2.7095    &  - \\
   3  &  1.0076    &  1.3743    &  2.0000    & 3.4094  \\
   4  &  1.0021    &  1.2514    &  1.7565    & 2.4190     \\
   5  &  1.00072   &  1.1637    &  1.6309    & 2.0869 \\
   6  &  1.00027   &  1.1039    &  1.5533    & 1.9201 \\
  \hline
\end{tabular}
\end{center}

\begin{rem}
It remains an interesting project to try to compute the precise value of $\gamma(B_{m})$ or to at
least determine better upper and lower bounds on $\gamma(B_{m})$ than the rudimentary ones that we have given
in this paper. In the event that computing the precise value of $\gamma(B_{m})$ proves intractable
(which can't be ruled out), estimating the rate at which $\gamma(B_{m})$ decreases to $1$ as $m \rightarrow \infty$
would be of particular interest.
\end{rem}

\bibliographystyle{amsalpha}

\end{document}